\documentclass[reqno,article,12pt,twosided,a4paper]{amsart}
\usepackage{amssymb,amsfonts,latexsym,mathrsfs}
\usepackage[all]{xy}
\usepackage{array}

\newtheorem{theorem}{Theorem}[section]
\newtheorem{lemma}[theorem]{Lemma}

\newtheorem{corollary}[theorem]{Corollary}

\theoremstyle{definition}
\newtheorem{definition}[theorem]{Definition}

\newtheorem{remark}[theorem]{Remark}

\usepackage{amstext}
\usepackage{amssymb}
\usepackage{amsfonts}
\usepackage{amsthm}
\usepackage{amsopn}
\usepackage{amsbsy}
\usepackage{layout}

\newcommand{\Z}{\mathbb{Z}}
\newcommand{\italic}{\textit}

\newcommand{\tensor}{\otimes}

\newcommand{\del}{\partial}
\newcommand{\mapsonto}{\twoheadrightarrow}

\newcommand{\ndivide}{\nmid}

\begin{document}
\title[The Hopf Schur group]{Every Central simple algebra is Hopf Schur}
\author{Ehud Meir}
\date{December 31, 2009}
\maketitle

\begin{abstract} We show that every central simple algebra $A$ over a field $k$ is
Brauer equivalent to a quotient of a finite dimensional Hopf algebra
over the same field (that is- $A$ is Hopf Schur). If the
characteristic of the field is zero, or if the algebra has a Galois
splitting field of degree prime to the characteristic of $k$, we can
take this Hopf algebra to be semisimple. We also show that if $F$ is
any finite extension of $k$, then $F$ is a quotient of a finite
dimensional Hopf algebra over $k$. We use it in order to show why
the algebric closeness assumption is necessary in a weak form of
Kaplansky's tenth conjecture, due to Stefan.\end{abstract}

\begin{section}{Introduction}\label{intro}
Let $k$ be a field. In \cite{HS} we asked what central simple $k$-algebras
can we get (up to Brauer equivalence) as quotients of
finite dimensional Hopf algebras over $k$. We called such algebras Hopf Schur
algerbas, and we defined the Hopf Schur group of $k$, $HS(k)$, as
the subgroup of $Br(k)$ which contains all classes of Hopf Schur
algebras. This is analogue to the definition of the Schur group,
$S(k)$, which contains all classes of central simple algebras which
are quotients of finite dimensional group algebrs, and also to the
projective Schur group, $PS(k)$, which contains all classes of
central simple algebras which are quotients of finite dimensional
twisted group algebrs. Since any group algebra is a Hopf algebra,
clearly $S(k)<HS(k)$. The Schur group is a ``small'' subgroup of
$Br(k)$. It is known by a theorem of Brauer (see \cite{IS}) that any
element in $S(k)$ has a cyclotomic splitting field, and therefore if
$k$ contains all roots of unity then $S(k)=0$, whereas $Br(k)$ may
be large (e.g. $k=\mathbb{C}(x_1,x_2,\ldots,x_n)$, $n\leq 2$, see
\cite{FS}). We refer the reader to \cite{J} and to \cite{Y} for a
comprehensive acount on the Schur group. In \cite{HS} it was shown
that $HS(k)$ might be much bigger than $S(k)$. The authors have
proved that $PS(k)<HS(k)$. The projective Schur group is already a
much bigger subgroup of the Brauer group. It was conjectured that
$PS(k)=Br(k)$ and this is indeed the case for many interesting
fields (e.g. number fields). It was disproved, however, by Aljadeff
and Sonn in \cite{AS}. In \cite{HS} it was also proved that any
product of cyclic algebras is in $HS(k)$, and there is a conjecture
which says that this is all of the Brauer group. In this paper we
will prove the following:
\begin{theorem}\label{mainth} Any $k$-central simple algebra is
Bruaer equivalent to a quotient of a finite dimensional Hopf algebra
(that is- $Br(k)=HS(k)$). If a $k$-central simple algebra $A$ has a
Galois splitting field $F$ such that $char(k)\ndivide |F:k|$, then
$A$ is Brauer equiavelnt to a quotient of a semisimple finite
dimensional Hopf algebra.\end{theorem}

Since division algebras arise naturally as Endomorphism rings of
simple representations, we have the following:
\begin{corollary} Let $D$ be a $k$-central division
algebra. Then there is a finite dimensional Hopf algebra $H$, and a
simple representation $V$ of $H$ such that $End_H(V)\cong D$. If in
addition $D$ has a splitting field $F$ such that $char(k)\ndivide
|F:k|$, we can take $H$ to be semisimple.\end{corollary}
\begin{proof} Let $D$ be a $k$-central division algebra.
By the above theorem, we have a Hopf algebra $H$ (seimsimple in case
$char(k)\ndivide |F:k|$, where $F$ is some splitting field of $D$),
and an algebra map $\pi:H\mapsonto M_n(D^{op})$ for some $n$, where
$D^{op}$ is the algebra $D$ with opposite multiplication. Then $V =
(D^{op})^n$ is a representation of $H$ via $\pi$, and it is easy to
see that $End_H(V)\cong D$.\end{proof}

In the course of the proof of Theorem \ref{mainth} we will consider, in section
\ref{function}, forms of the function algebra of a finite group (i.e. the dual of a finite group algebra). As
a result, we will see that there might be an infinite number of
non-isomorphic semisimple and cosemisimple Hopf algebras over $k$ of a given
dimension, if $k$ is not algebraically closed. In \cite{St} Stefan have proved that for a given number $n$, there are only finitely many isomorphism classes of semisimple and cosemisimple Hopf algebras of dimension $n$ over an algebraically closed field. We therefore conclude that the algebraic closeness assumption in Stefan's Theorem is necessary. Stefan's Theorem is a weaker form of Kaplansky's tenth conjecture, which states that for a given number $n$ there are only finitely many isomorphism classes of (not necessarily semisimple and cosemisimple) Hopf algebras of dimension $n$ over an algebraically closed field. Kaplansky's tenth conjecture was disproved by Andruskiewitsch and Schneider (see \cite{AS1}) by Beattie, Dascalescu and Grunenfelder (see \cite{BDG}), and by Gelaki (see \cite{Ge}) (and so, also the semisimplicity and the cosemisimplicity of the Hopf algebra is necessary in Stefan's Theorem).

In order to prove Theorem \ref{mainth} we will revise the construction from
\cite{HS}. We will first review the relevant notions from Galois
descent in Section \ref{descent}. The main idea we shall use from
Galois descent is that a Hopf algebra over $k$ is the same thing as
a Hopf algebra over a Galois extension of $k$ together with a
certain ``semilinear'' action. In section \ref{function} we use
Galois descent in order to understand the forms of function algebras on finite groups.
We use this in order to prove that any finite field
extension of $k$ is a quotient of a Hopf algebra over $k$. Besides Galois descent,
we shall also need to consider semidirect products of Hopf algebras.
The relevant details will be discussed in Section \ref{semidirect}.
Finally, in Section \ref{main} we will make a reduction and show
that it is enough to prove that any crossed product algebra for
which the values of the cocycle are all roots of unity is a quotient
of a Hopf algebra. We then show how we can construct such a Hopf
algebra, using Galois descent and semidirect products. \end{section}

\begin{section}{Galois descent}\label{descent}
We shall need to use a very small portion of the descent theory in
here. The reader is referred to \cite{GS} for more comprehensive
treatment. Let $L/k$ be a Galois extension with a galois group $G$.
\begin{definition}
Let $V$ be a vector space over $L$. A $G$-semilinear action on $V$
is an action of $G$ on $V$ as a $k$-vector space such that $g(x\cdot
v)= g(x)\cdot g(v)$, for every $g\in G$, $x\in L$ and $v\in
V$.\end{definition} We shall simply say ``semilinear action''
instead of ``$G$-semilinear action'' if the group is clear from the
context. The typical example we should have in mind for a semilinear
action is the following: suppose that $\hat{V}$ is a vector space
over $k$. Then $V=L\otimes_k\hat{V}$ is a vector space over $L$, and
we have a semilinear action given by $g\cdot (x\otimes v) =
g(x)\otimes v$. In descent theory it is proved that every semilinear
action is of this form. We shall need to use the following lemma,
whose proof is rather easy:
\begin{lemma} Let $H$ be a Hopf algebra over $L$. Suppose that $H$ has a
semilinear action which respects the Hopf structure (i.e. it
commutes with all the Hopf algebra structure maps). Then $H^G$ is a
Hopf algebra over $k$ (the structure maps are just the restrictions of the structure maps of $H$).\end{lemma} So in order to construct Hopf
algebras over $k$, we can construct Hopf algebras over $L$ together
with a semilinear action. In descent theory, all the semilinear
actions are classified by a certain nonabelian cohomology group. In
our case it would be easier to find semilinear actions directly.
\end{section}

\begin{section}{Semilinear actions on function algebras}\label{function}
Let $L,k$ and $G$ be as before, and let $T$ be any finite group. We
consider the function algebra (which is also the dual of the group
algebra of $T$, $L[T]=(LT)^*$). This is the $L$-algebra of all the
functions from $T$ to $L$. This algebra has a basis of simple
idempotents $\{e_t\}_{t\in T}$, where $e_t(s) = \delta_{t,s}$ for
$t,s\in T$. This algebras also has a Hopf structure given by
$\Delta(e_t) = \sum_{rs=t}{e_r\otimes e_s}$. In this section we
shall describe the semilinear actions of $G$ on $L[T]$ which
respects the Hopf structure, and how the corresponding invariant
algebras look like.
\begin{lemma} There is a one to one correspondence between semilinear
actions on $L[T]$ and homomorphism $G\rightarrow Aut(T)$ (i.e. actions of $G$ on $T$).\end{lemma}
\begin{remark} The reader who is familiar with descent theory will
notice that this correspondence is exactly the correspondence between $k$-forms of $k[T]$ and
$H^1(G,Aut(T))$, where $G$ acts trivially on $Aut(T)$.\end{remark}
\begin{proof} The correspondence is given in the following way:
for $\phi:G\rightarrow Aut(T)$ we have the semilinear action
$g_{\phi}\cdot xe_t = g(x)e_{\phi(g)(t)}$ for $g\in G$, $x\in L$ and
$t\in T$. Any semilinear action is of this form due to the following
reason: since the action is by algebra automorphisms, every $g\in G$
permutes the set of simple idempotents $\{e_t\}_{t\in T}$, and so
acts on $T$. The fact that the action preserves the coalgebra
structure means that this permutation is an automorphism of $T$ as a
group.\end{proof}
\begin{remark}
In Section 3 of \cite{HS} we have described explicitly a specific
form of a specific function algebra. The situation there was the
following: we had an abelian Galois extension $L/k$ with an
(abelian) Galois group $G$. We have constructed a form for the
function algebra $k[Z_2\ltimes G]$, where $\Z_2$ acts on $G$ by
inversions. If we denote the generator of $\Z_2$ by $\sigma$, then
the map $G\rightarrow Aut(T)$ which corresponds to this form, is the
map which sends $g\in G$ to the automorphism $\phi_g$ of $Z_2\ltimes
G$ which fixes $G$ pointwise, and sends $\sigma$ to $\sigma
g$.\end{remark} Let us describe, for a given $\phi:G\rightarrow
Aut(T)$, what would be the algebra of invariants $(L[T])^G$. We will
describe only the algebra structure and not the coalgebra structure,
since it will be enough for our purposes. Let $a = \sum_{t\in
T}a_te_t\in L[T]^G$. It is easy to see that the fact that $a$ is
invariant is equivalent to the fact that $g(a_t)=a_{\phi(g)(t)}$ for
every $g\in G$ and $t\in T$. In particular $a_t\in L^{stab(t)}$,
where by $stab(t)$ we denote the stabilizer of $t$ in $G$ with
respect to the action $\phi$. If we fix representatives of the
different orbits $t_1,\ldots,t_m$, then we have an isomorphism of
algebras
$$L^{stab(t_1)}\oplus L^{stab(t_2)}\oplus\ldots\oplus L^{stab(t_m)}\rightarrow (L[T])^G$$ given by
$$x\mapsto \sum_{g\in G/stab(t_i)}g(x)e_{\phi(g)t_i}$$ for $x\in L^{stab(t_i)}$.
Notice in particular that all the fields $L^{stab(t_i)}$ are
quotient of the $k$-Hopf algebra $(L[T])^G$. Can we get any subfield of
$L$ that way? the answer is yes. Using Galois correspondence, any
subfield of $L$ is of the form $L^H$, for some $H<G$. Therefore we
need to prove that for every $H<G$ we have a group $T$ and an action
of $G$ on $T$ such that $T$ contains an element $t$ for which
$stab(t)=H$. Let us take $T=\Z_2 G/H$, the vector space over $\Z_2$
with the coset space $G/H$ as a basis. the group $G$ acts from the
left on $G/H$ and therefore also on $T$. It is clear that if we take
$t=H$ (the trivial coset), then $stab(t)=H$ as required. Since $L$
was an arbitrary Galois extension of $k$, and any finite extension
of $k$ is contained in its Galois closure, we have proved the
following:
\begin{theorem} Let $F/k$ be any finite extension. Then there
is a semisimple commutative Hopf algebra $H$ over $k$ such that $F$
is a quotient of $H$.\end{theorem} Now assume that we have an
infinite number of nonisomorphic Galois extensions of $k$ with the
same Galois group $G$ which satisfies the condition $char(k)\ndivide |G|$ (e.g. $k=\mathbb{Q}$ and $G=\Z_2$). Then any Galois extension $L$ of $k$ with a Galois group $G$ is a quotient of a twisted form of $k[\Z_2 G]$. The algebra $k[\Z_2 G]$ (and each of its forms) is semisimple and cosemisimple (by the assumption on the order of $G$).
It is easy to see that by considering all these forms, we get an infinite number of nonisomorpic commutative
semisimple and cosemisimple Hopf algebras of dimension $2^{|G|}$, as was claimed in Section \ref{intro}.
\end{section}

\begin{section}{A semidirect product}\label{semidirect}
In this section we will construct semidirect products of Hopf
algebras over $k$. Let $N$ and $T$ be groups such that $N$ acts on
$T$ by group automorphisms, i.e. we have a homomorphism
$\psi:N\rightarrow Aut(T)$. Given $\psi$, we will construct a Hopf
algebra $X_{\psi}$ which would be the semidirect product of $kN$ and
$k[T]$. As a coalgebra, $X_{\psi}$ would be $k[T]\otimes_k kN$. The
product in $X_{\psi}$ is given by the rule $$e_{t_1}\otimes h_1\cdot
e_{t_2}\otimes h_2=\delta_{t_1,\psi(h_1)(t_2)}e_{t_1}\otimes
h_1h_2.$$ In other words- $k[T]$ and $kN$ are subalgebras of
$X_{\psi}$, and $n\in N$ acts by conjugation on $e_t$ via $\psi$.
The algebra $X_{\psi}$ is a Hopf algebra. It is the bicrossed
product of the Hopf algebras $kN$ and $k[T]$. For the definition of
bicrossed products in general, see \cite{K}. We have an algebra map
$X_{\psi}\mapsonto kN$. Therefore, if $kN$ is not semisimple, the
same is true for $X_{\psi}$. On the other hand, a brief calculation
shows that if $kN$ is semisimple, then so is $X_{\psi}$. Using
Maschke's Theorem, we have the following
\begin{lemma}\label{ss} The algebra $X_{\psi}$ is semisimple if and only if
$char(k)\ndivide |N|$.\end{lemma}\end{section}

\begin{section}{A proof of Theorem \ref{mainth}}\label{main}
In this section we will show that $HS(k)=Br(k)$. As before, let $L/k$
be a Galois extension with a Galois group $G$. Let $[\alpha]\in
H^2(G,L^*)$, where the action of $G$ on $L^*$ is the Galois action.
\begin{definition} We say that the cocycle $\alpha$ is
\italic{finite} if all its values are roots of
unity.\end{definition} Note that this definition depends on the
particular cocycle $\alpha$, and not just on its cohomology class
$[\alpha]$. We would like to prove that any central simple algebra
is Brauer equivalent to a quotient of a finite dimensional Hopf
algebra. We shall do this in the following way: Since any central
simple algebra is known to be Brauer equivalent to a crossed product
algebra, we will prove first that any crossed product algebra
$L^{\alpha}_tG$ is Brauer equivalent to a crossed product algebra
$K^{\beta}_t N$ in which the cocycle $\beta$ is finite. We then
prove that any such crossed product is a quotient of a finite
dimensional Hopf algebra.
\begin{lemma}\label{finite}
The crossed product $L^{\alpha}_t G$ is Brauer equivalent to a
crossed product algebra $K^{\beta}_t N$ with $\beta$
finite.\end{lemma}
\begin{proof} Let $m$ be the order of $\alpha$ in $H^2(G,L^*)$. Suppose
that $f:G\rightarrow L^*$ satisfies $$\alpha^m(g_1,g_2) =
\del f (g_1,g_2)=f(g_1)f(g_2)f^{-1}(g_1g_2),$$ for $g_1,g_2\in G$. Let $K$ be a
Galois extension of $L$ which contains elements $r_{g}$, for every
$g\in G$, for which the equations $r_{g}^m = f(g)$ hold. If we
denote by $N$ the Galois group of $K$ over $k$, we have an onto map
$\pi:N\mapsonto G$. Define a two cocycle $\beta\in H^2(N,K^*)$ by
$$\beta(h_1,h_2)
=\alpha(\pi(h_1),\pi(h_2))r^{-1}_{\pi(h_1)}r^{-1}_{\pi(h_2)}r_{\pi(h_1h_2)}.$$
A direct calculation shows that all the values of $\beta$ are $m$-th
roots of unity, and so $\beta$ is finite. The cocycle $\beta$ is
cohomologous to $\inf_G^N(\alpha)$. By Brauer theory we thus know
that the central simple algebras $L^{\alpha}_tG$ and $K^{\beta}_t N$
are Brauer equivalent, as required.\end{proof} We will therefore
assume from now, without loss of generality, that all the values of
$\alpha$ are roots of unity. Recall that $L^{\alpha}_t G$ has an
$L$-basis $\{U_g\}_{g\in G}$, and the multiplication in
$L^{\alpha}_t G$ is given by $xU_g\cdot yU_h =
xg(y)\alpha(g,h)U_{gh}$. By considering the subgroup of
$L^{\alpha}_tG$ generated by the $U_g$'s, we get an extension of
\textbf{finite} groups $$1\rightarrow \mu\rightarrow
\widehat{G}\rightarrow G\rightarrow 1,$$ where $\mu$ is the
\textbf{finite} subgroup of $L^*$ generated by all elements of the
form $\alpha(g_1,g_2)$, for $g_1,g_2\in G$. Let $T$ be the group
$\Z_2 G$, the vector space over $\Z_2$ with basis $G$
(multiplication in $T$ is just addition of vectors). We have a
natural action of $G$ (and thus of $\widehat{G}$, using the map
$\widehat{G}\rightarrow G$) on $T$ by left multiplication. If we
denote the action of $\widehat{G}$ on $T$ by $\psi$, we can
construct the semidirect product $X_{\psi}$, which would be
$k[T]\otimes_k k\widehat{G}$ as a vector space. Notice that by Lemma
\ref{ss} $X_{\psi}$ is semisimple if and only if $char(k)\ndivide
|G|$ (it is easy to see, by the proof of Lemma \ref{finite} and by
the fact that the order of $\alpha$ in $H^2(G,L^*)$ divides the
order of $G$, that the prime divisors of $|\mu|$ are also prime
divisors of $|G|$). Now consider the induced $L$-Hopf algebra
$X_L=L\otimes_k X_{\psi}$. We have an action of $G$ on $T$ not only
from the left but also from the right by multiplication. We define
an action of $G$ on $X_L$ via $$g\star (l\otimes e_t\otimes h) =
g(l)\otimes e_{t\cdot g^{-1}}\otimes h.$$ We claim the following
\begin{lemma} The action $\star$ is a semilinear action of $G$ on $X_L$.\end{lemma}
\begin{proof} This is a straightforward verification.
The crux of the proof is the fact that the two actions of $G$ on $T$
from the left and from the right commute with each other.\end{proof}

Consider now the $k$-Hopf algebra $H=(X_L)^G$, where we are taking
invariants with respect to the $\star$ action. We claim that we have
an onto map $H\mapsonto L^{\alpha}_t G$. To see why this is true,
notice first that we have a decomposition of $H$ as a direct sum of
two sided ideals $H=H_1\oplus H_2$. The ideal $H_1$ is the
intersection of $(X_L)^G$ with the $L$ subspace spanned by all
$1\otimes e_{t}\otimes\hat{g}$, where $\hat{g}\in \widehat{G}$, and
$t\in G$ (We can consider $G$ as the subset of $T$ which contains
the basis elements). The ideal $H_2$ is the interesection of
$(X_L)^G$ with the $L$ subspace spanned by all $1\otimes
e_{t}\otimes\hat{g}$, where $\hat{g}\in \widehat{G}$ and $t\notin
G$. So it is enough to prove that we have an onto map $H_1\mapsonto
L^{\alpha}_t G$ of algebras. It is easy to see that $H_1$ is spanned
by elements of the form $\sum_{g\in G}g(l)\otimes
e_{g^{-1}}\otimes\hat{u}$, where $l\in L$ and $\hat{u}\in \widehat{G}$. So
$H_1$ has $k\widehat{G}$ and $L$ as subalgebras. The inclusion of
$L$ as a subalgebra of $H_1$ is given by $l\mapsto \sum_{g\in
G}g(l)\otimes e_{g^{-1}}\otimes 1$, and the inclusion of $k\widehat{G}$
as a subalgebra of $H_1$ is given by $x\mapsto 1\otimes 1\otimes x$.
It is easy to see that $H_1$ is the semidirect product of
$k\widehat{G}$ with $L$, where the action of $\widehat{G}$ on $L$ is
given by the Galois action of $G$ on $L$ (recall that $G$ is a
quotient of $\widehat{G}$). We would like to define now an algebra
map from $H_1$ onto $L^{\alpha}_t G$. To do this, recall that the
group $\widehat{G}$ was constructed as a subgroup of the group of
invertible elements in $L^{\alpha}_t G$. We therefore have a natural
inclusion map $\widehat{G}\stackrel{i}{\rightarrow} L^{\alpha}_tG$. We can now
define an algebra homomorphism $\Phi:H_1\rightarrow L^{\alpha}_t G$
in the following way: since $H_1$ is the semidirect product of $k\widehat{G}$ and $L$, it is enough to define $\Phi$ on products of the form $l\hat{g}$ where $l\in L$ and $\hat{g}\in \widehat{G}$. We define $\Phi(l\hat{g})=li(\hat{g})$. It is easy to see that $\Phi$ is a well defined onto algebra map. This proves
Theorem \ref{mainth}.
\end{section}

\end{document}